\documentclass[12pt,a4paper,english]{article}
\usepackage[T1]{fontenc}
\usepackage{amsmath,amsthm,amssymb,bbm,color}
\usepackage{mathtools}
\usepackage[english]{babel}
\usepackage[active]{srcltx}
\usepackage{fancyhdr}
\allowdisplaybreaks

\newtheorem{lemma}{Lemma}[section]
\newtheorem{proposition}{Proposition}[section]
\newtheorem{theorem}{Theorem}[section]
\newtheorem{corollary}{Corollary}[section]
\newtheorem{remark}{Remark}[section]

\newcommand{\norm}[1]{\left\lVert#1\right\rVert}
\newcommand{\normP}[1]{\left\lVert#1\right\rVert_{L_2(\PP)}}
\newcommand{\normmP}[1]{\left\lVert#1\right\rVert_{L_2(\mm\otimes\PP)}}

\newcommand{\one}{\mathbbm{1}}
\newcommand{\PP}{{\mathbbm{P}}}
\newcommand{\R}{\mathbbm{R}}
\newcommand{\E}{\mathbbm{E}}
\newcommand{\ud}{{\mathrm{d}}}

\newcommand{\mm}{\mathbbm{m}}
\newcommand{\Lmn}{{L_2\left( \mm^{\otimes n} \right)}}
\newcommand{\DD}{\mathbbm{D}_{1,2}}
\newcommand{\DS}{\mathcal{S}}

\newcommand{\itemm}[1]{\item [{\rm (#1)} ]}

\newcommand{\om}{{\omega}}
\newcommand{\Om}{{\Omega}}

\newcommand{\s}{{\sigma}}

\newcommand{\vph}{{\varphi}}
\newcommand{\la}{{\lambda}}

\newcommand{\brac}{\left(}
\newcommand{\kets}{\right)}
\newcommand{\sqr}{\left[}
\newcommand{\brackets}{\right]}

\newcommand{\equa}{\begin{equation}}
\newcommand{\tion}{\end{equation}}

\newcommand{\bor}{\mathcal{B}}
\newcommand{\ftn}{{\cal F}}

\numberwithin{equation}{section}

\begin{document}

\title{A note on Malliavin smoothness on the L\'{e}vy space}

\author{Eija Laukkarinen\\
   Department of Mathematics and Statistics \\   P.O.Box 35, FI-40014 University of Jyv\"{a}skyl\"{a}, Finland \\
 \texttt{  eija.laukkarinen@jyu.fi}  }

\maketitle


\begin{abstract}
 We consider Malliavin calculus based on the It\^{o} chaos decomposition of square integrable random variables on the L\'{e}vy space. 
 We show that when a random variable satisfies a certain measurability condition, its differentiability and fractional differentiability 
 can be determined by weighted Lebesgue spaces. The measurability condition is satisfied for all random variables if the underlying L\'evy process is 
 a compound Poisson process on a finite time interval.
\end{abstract}

{\parindent0em {\em Keywords:} L\'{e}vy process, Malliavin calculus, interpolation}\\
{\parindent0em \em AMS2010 Subject Classification:}
60G51, 
60H07 

\section{Introduction}

  One extension of Malliavin calculus from the Brownian motion to general L\'{e}vy processes was made using the It\^{o} chaos 
  decomposition on the $L_2$-space over the L\'evy space. This approach was used for
   instance by Nualart and Vives \cite{nualart-vives}, Privault \cite{privault_extension}, Benth, Di Nunno, L{\o}kka, {\O}ksendal and Proske 
   \cite{benth-dinunno-lokka-oksendal-proske}, Lee and Shih \cite{lee-shih_product_formula},  Sol\'e, Utzet and Vives
   \cite{sole-utzet-vives} and Applebaum
  \cite{applebaum2}. 
  
  The wide interest in Malliavin calculus for L\'{e}vy processes in stochastics and 
  applications motivates the study of an accessible characterization
  of differentiability and fractional differentiability.
  Fractional differentiability is defined by real interpolation between the Malliavin Sobolev space $\DD$ and $L_2(\PP)$ and
  we recall the definition in Section \ref{section:fractional} of this paper. 
  Geiss and Geiss \cite{geiss-geiss} and Geiss and Hujo \cite{geiss-hujo} have shown that Malliavin differentiability and 
  fractional differentiability 
  are in a close connection to discrete-time approximation of certain stochastic integrals when the underlying process is a (geometric)
  Brownian motion. Geiss et al. \cite{geiss-geiss-laukkarinen} proved that this applies also to L\'{e}vy processes with jumps.
  These works assert that knowing the parameters of 
  fractional smoothness allow to design discretization time-nets
  such that the optimal approximation rate can be achieved. 
  For details, see \cite{geiss-geiss}, \cite{geiss-hujo} and \cite{geiss-geiss-laukkarinen}.

  Steinicke \cite{steinicke} and Geiss and Steinicke \cite{geiss-steinicke} take advantage of the fact that any random variable $Y$ on 
  the L\'evy space can be represented as a functional $Y = F(X)$ of the L\'evy process $X$, where $F$ is a real valued measurable mapping 
  on the Skorohod space of right continuous functions. Let us restrict to the case that $F(X)$ only depends on the jump part of $X$. 
  Using the corresponding result from 
  Sol\'e, Utzet and Vives \cite{sole-utzet-vives} and Al\`os, Le\'on and Vives \cite{alos-leon-vives} on the canonical space, Geiss and Steinicke
  \cite{geiss-steinicke} show that the condition $F(X)\in\DD$ is equivalent with
  $$\iint_{\R_+\times\R}\E\left[ \left(F(X+x\one_{[t,\infty)}) - F(X) \right)^2 \right] \ud t\nu(\ud x) < \infty, $$
  where $\nu$ is the L\'evy measure of $X$.
  On the other hand one gets from Mecke's formula \cite{mecke} that
  $$\iint_A \E[F( X+x\one_{[t,\infty)} )]\ud t\nu(\ud x) = \E[N(A)F(X)]$$
  for any nonnegative measurable $F$ and any $A\in\bor([0,\infty)\times\R\setminus\{0\})$, 
  where $N$ is the Poisson random measure associated with $X$  as in Section \ref{section:preliminaries}. 
  These results raise the following questions: when can Malliavin 
  differentiability be described using a weight function such as $N(A)$, and is there a weight function for fractional differentiability?

 In this paper  we search for weight functions $\Lambda$ and measurability conditions on $Y$ such that the criteria  
   \begin{equation}\|Y \Lambda\|_{L_2(\PP)} < \infty  \label{equation:weight_criteria}\end{equation}
  describes the smoothness of $Y$. We begin by recalling the orthogonal It\^{o} 
  chaos decomposition
 $$Y = \sum_{n=0}^\infty I_n(f_n)$$
 on $L_2(\PP)$ and the Malliavin Sobolev space
 $$\DD = \left\{ Y\in L_2(\PP): \|Y\|_{\DD} = \sum_{n=0}^\infty (n+1) \|I_n(f_n)\|_{L_2(\PP)}^2 < \infty \right\}$$ 
 in Section \ref{section:preliminaries}.
 Then,  in Section \ref{section:differentiability}, we obtain an equivalent condition for Malliavin differentiability. The assertion is that 
 $$Y\in\DD \text{ if and only if } \normP{ Y \sqrt{N(A) +1}} < \infty,$$
  whenever $Y$ is measurable with respect to $\ftn_A$, the
 completion of the sigma-algebra generated by $\left\{ N(B) : B\subseteq A, \,B\in\bor([0,\infty)\times\R)\right\}$ 
  and the set $A\in\bor([0,\infty)\times\R\setminus\{0\})$ satisfies $\E[N(A)]<\infty$.

 Section \ref{section:fractional} treats fractional differentiability and our aim is to adjust the weight function $\Lambda$  so that the 
 condition \eqref{equation:weight_criteria} describes 
 a given degree of smoothness. We recall the $K$-method of real interpolation which we 
 use to determine the interpolation spaces $(L_2(\PP),\DD)_{\theta,q}$ for $\theta\in(0,1)$ and $q\in[1,\infty]$.  These spaces are intermediate
 between $\DD$ and $L_2(\PP)$.
 We show that when $Y$ is $\ftn_A$-measurable and $\E[N(A)]<\infty$, then $Y$ has fractional differentiability of order $\theta$ for $q=2$ 
 if and only if
$$\normP{ Y\sqrt{ N(A) +1}^{\,\theta} } < \infty.$$
%

\section{Preliminaries} \label{section:preliminaries}

Consider a L\'evy process $X = (X_t)_{t\geq0}$ with c\`{a}dl\`{a}g paths on a
complete probability space $({\Omega},{\cal F},\mathbbm{P})$, where ${\cal F}$ is the completion of the sigma-algebra generated by $X$.
The L\'{e}vy-It\^{o} decomposition states that there exist $\gamma\in\mathbbm{R}$, $\sigma\geq0$, a standard Brownian motion $W$ and 
a Poisson random measure $N$ on $\mathcal{B}([0,\infty)\times\mathbbm{R})$ such that
\[X_t=\gamma t + \sigma W_t + \iint_{(0,t]\times \{ |x|>1\}} x N(\mathrm{d} s,\mathrm{d} x) 
+  \iint_{(0,t]\times\left\{0<|x|\leq1\right\}}x\tilde{N}(\mathrm{d} s,\mathrm{d} x)\]
 holds for all $t\geq0$ a.s.
Here $\tilde{N}(\mathrm{d} s,\mathrm{d} x) = N(\mathrm{d} s,\mathrm{d} x)-\mathrm{d} s\nu(\mathrm{d} x)$ is the compensated Poisson random measure and 
$\nu:\mathcal{B}(\mathbbm{R})\to[0,\infty]$ is the L\'{e}vy measure of $X$ satisfying $\nu(\{0\})=0$,
$\int_\mathbbm{R} (x^2\wedge1)\nu(\mathrm{d} x)<\infty$ and $\nu(B)=\mathbbm{E} \left[ N((0,1]\times B) \right]$ when $0\not\in \bar{B}$.
The triplet $(\gamma,\sigma,\nu)$ is called the L\'evy triplet.

Let us recall the It\^{o} chaos decomposition from \cite{ito}:
Denote $\mathbbm{R}_+ := [0,\infty)$. We consider the following measure $\mathbbm{m}$ defined as 
\begin{eqnarray*}
 \mathbbm{m}:\mathcal{B}(\mathbbm{R}_+\times\mathbbm{R})\to[0,\infty],& 
 & \mathbbm{m}(\mathrm{d} s,\mathrm{d} x):= \mathrm{d} s \sqr {\sigma}^2 \delta_0(\mathrm{d} x)  + x^2 \nu(\mathrm{d} x) \brackets.
\end{eqnarray*}
For sets $B\in\mathcal{B}(\mathbbm{R}_+\times\mathbbm{R})$ such that $\mathbbm{m}(B) < \infty$, a random measure $M$ is defined by
\[M(B) := \sigma \int_{\left\{ s\in\mathbbm{R}_+:(s,0)\in B\right\}} \mathrm{d} W_s + \lim_{n\to\infty}\iint_{\left\{(s,x)\in B: \frac{1}{n} < |x| < n\right\}} x\ \tilde{N}(\mathrm{d} s,\mathrm{d} x),\]
where the convergence is taken in $L_2(\mathbbm{P}):=L_2({\Omega},{\cal F},\mathbbm{P})$.
The random measure $M$ is independently scattered and it holds that $\mathbbm{E}
M(B_1) M(B_2) = \mathbbm{m}(B_1 \cap B_2)$ for all $B_1,B_2\in\mathcal{B}(\mathbbm{R}_+\times\mathbbm{R})$ with
$\mathbbm{m}(B_1)<\infty$ and $\mathbbm{m}(B_2)<\infty$.

For $n=1,2,\ldots$ write 
\[ \Lmn = L_2 \left((\mathbbm{R}_+\times\mathbbm{R})^n, \mathcal{B}(\mathbbm{R}_+\times\mathbbm{R})^{\otimes n}, \mathbbm{m}^{\otimes n}\right)\] 
and set $L_2 \left(\mathbbm{m}^{\otimes 0} \right):=\mathbbm{R}$. A function $f_n:(\mathbbm{R}_+\times\mathbbm{R})^n\to\mathbbm{R}$ is said to
be symmetric, if it coincides with its symmetrization
$\tilde{f}_n$,
\[\tilde{f}_n((s_1,x_1),\ldots,(s_n,x_n))=\frac{1}{n!}\sum_{\pi}f_n\left(  \left( s_{\pi(1)},x_{\pi(1)} \right),\ldots, \left( s_{\pi(n)},x_{\pi(n)} \right) \right),\]
where the sum is taken over all permutations
$\pi:\{1,\ldots,n\}\to\{1,\ldots,n\}$.

We let $I_n$ denote the multiple integral of order $n$ defined by It\^{o} \cite{ito} and shortly recall the definition. 
For pairwise disjoint 
$B_1,\ldots, B_n\in\mathcal{B}(\mathbbm{R}_+\times\mathbbm{R})$ with $\mathbbm{m}(B_i)<\infty$ the 
integral of $\mathbbm{1}_{B_1}\otimes\cdots\otimes\mathbbm{1}_{B_n}$ 
is defined by
\begin{equation} I_n\left( \mathbbm{1}_{B_1}\otimes\cdots\otimes\mathbbm{1}_{B_n}\right) := M(B_1)\cdots M(B_n). \label{equation:multiple_integral}\end{equation}
It is then extended to a linear and continuous operator $I_n:\Lmn\to L_2(\mathbbm{P})$. We let $I_0(f_0):=f_0$ for $f_0\in\mathbbm{R}$. 
For the multiple integral we have
\begin{equation}\label{equation:inner_product_L_2}
I_n(f_n)=I_n(\tilde{f}_n) \text{ and } \mathbbm{E} \sqr I_n(f_n)I_k(g_k) \brackets 
=\begin{cases}
        0, & \text{ if }n\neq k\\
       n! \left( \tilde{f_n},\tilde{g_n}\right)_{L_2(\mathbbm{m}^{\otimes n})}, & \text{ if }n=k
       \end{cases}
\end{equation}
for all $f_n\in \Lmn$ and $g_k\in L_2\left(\mathbbm{m}^{\otimes k}\right)$. 

According to \cite[Theorem 2]{ito}, for any $Y\in L_2(\PP)$ there exist functions $f_n\in \Lmn$,
$n=0,1,2,\ldots,$ such that
\[ Y=\sum_{n=0}^\infty
I_n(f_n) \quad \text{ in } L_2(\mathbbm{P})\]
and the functions $f_n$ are unique  in $L_2(\mm^{\otimes n})$ when they are chosen to be
symmetric. We have
\[ \norm{Y}^2_{L_2(\mathbbm{P})} = \sum_{n=0}^\infty n! \norm{\tilde{f}_n}^2_\Lmn.\]

We recall the definition of the Malliavin Sobolev space $\DD$ based on the
It\^{o} chaos decomposition. We denote by $\mathbbm{D}_{1,2}$ 
the space of all $Y = \sum_{n=0}^\infty
I_n(f_n) \in L_2(\mathbbm{P})$ such that
\[\|Y\|^2_{\mathbbm{D}_{1,2}} := \sum_{n=0}^\infty (n+1)! \norm{\tilde{f}_n}^2_\Lmn < \infty.\]
Let us write $L_2(\mathbbm{m}\otimes\mathbbm{P}) := L_2(\mathbbm{R}_+\times\mathbbm{R}\times{\Omega},
\mathcal{B}(\mathbbm{R}_+\times\mathbbm{R})\otimes{\cal F}, \mathbbm{m}\otimes\mathbbm{P})$ and define the
Malliavin derivative $D:\mathbbm{D}_{1,2}\to L_2(\mathbbm{m}\otimes\mathbbm{P})$  in the following way.
For $B_1,\ldots,B_n \in \bor(\R_+\times\R)$, which are pairwise disjoint and such that $\mm(B_i)< \infty$ for all $i=1,\ldots,n$, we let
\begin{align*}D_{t,x} I_n\left( \mathbbm{1}_{B_1}\otimes\cdots\otimes\mathbbm{1}_{B_n}\right) 
&= nI_{n-1}\left( \tilde{\mathbbm{1}_{B_1}\otimes\cdots\otimes\mathbbm{1}_{B_n}}(\cdot,(t,x))\right)\\
&:= \sum_{i=1}^n \prod_{j\neq i} M(B_j) \one_{B_i}(t,x). 
\end{align*}
It holds $\normmP{DI_n\left( \mathbbm{1}_{B_1}\otimes\cdots\otimes\mathbbm{1}_{B_n}\right)} 
= \sqrt{n}\normP{I_n\left( \mathbbm{1}_{B_1}\otimes\cdots\otimes\mathbbm{1}_{B_n}\right) }$ and
the operator is extended to $\left\{ I_n(f_n): f_n\in L_2(\mm^{\otimes n})  \right\}$ by linearity and continuity. For 
$Y = \sum_{n=0}^\infty I_n(f_n) \in \DD$ it then holds that
\[D_{t,x}Y := \sum_{n=1}^\infty n I_{n-1} \left(\tilde{f}_n(\cdot,(t,x))\right) \] 
converges in $L_2(\mathbbm{m}\otimes\mathbbm{P})$.

\begin{remark}\label{remark:inner_product_DD}
Note that also for any $u\in L_2(\mm\otimes\PP)$ one finds a chaos representasion $u=\sum_{n=0}^\infty I_n(g_{n+1})$, 
where the functions $g_{n+1} \in L_2\left(\mm^{\otimes(n+1)}\right)$
are symmetric in the first $n$ variables. For $u,v\in L_2(\mm\otimes\PP)$ with 
$u=\sum_{n=0}^\infty I_n(g_{n+1})$ and $v=\sum_{n=0}^\infty I_n(h_{n+1})$ it then holds
\begin{equation}\label{equation:inner_product_DD}
 (u,v)_{L_2(\mm\otimes\PP)} = \sum_{n=0}^\infty n!  \left(g_{n+1},h_{n+1} \right)_{L_2\left(\mm^{\otimes (n+1)}\right)}.
 \end{equation}
\end{remark}
For more information, see for example \cite{nualart-vives}, \cite{privault_extension}, \cite{benth-dinunno-lokka-oksendal-proske}, 
\cite{lee-shih_product_formula}, \cite{sole-utzet-vives} and
  \cite{applebaum2}. 

\section{Differentiability}
\label{section:differentiability}

We shall use the notation $\R_0 = \R\setminus\{0\}$. For $A\in\bor(\R_+\times\R_0)$ we denote by
$\ftn_A$ the completion of the sigma-algebra $ \sigma \left( N(B) : B\subseteq A \text{ and } B\in\bor(\R_+\times\R) \right) $. 
The following theorem implies that if $Y\in L_2(\PP)$ is $\ftn_A$-measurable and $\E[N(A)]<\infty$, then 
$Y\in\DD$ if and only if $\E[Y^2 N(A)]<\infty$.

\begin{theorem}\label{theorem:second_double_inequality}
 Let $A\in\bor(\R_+\times\R_0)$ be such that $ \E\left[N(A)\right] =  (\ud t \otimes \nu)(A) < \infty$ and 
 $Y\in L_2(\PP)$. 
 \begin{enumerate}
  \item If $Y\in\DD$, then $Y\sqrt{N(A)}\in L_2(\PP)$ and
        \begin{align}\label{equation:second_double_inequality1}
  \left| \normP{Y\sqrt{N(A)}} - \normP{Y} \sqrt{\E\sqr N(A)\brackets}  \right| 
   \leq \normmP{DY\one_A}.\end{align}
   \item If $Y\sqrt{N(A)}\in L_2(\PP)$ and $Y$ is $\ftn_A$-measurable, then $Y\in\DD$ and
     \begin{align}\label{equation:second_double_inequality2}
    \normmP{DY}
  \leq \normP{Y\sqrt{N(A)}} + \normP{Y}\sqrt{\E\sqr N(A)\brackets}.
 \end{align}
 \end{enumerate}
\end{theorem}

We denote by $\DS$ the set of random variables $Y$ such that there exists $m\geq 1$, $f\in C_c^\infty(\R^m)$ and  
$0\leq t_0 < t_1 < \cdots t_m < \infty$ such that
$$Y= f\brac X_{t_1} - X_{t_0},\ldots,X_{t_m} - X_{t_{m-1}}  \kets.$$

\begin{lemma}[Theorem 4.1, Corollaries 4.1 and 3.1 in \cite{geiss-laukkarinen}]\label{lemma:smooths_dense_in}\hfill

\begin{itemize}
 \itemm{a} $\DS$ is dense in $\DD$ and $L_2(\PP)$.\label{lemma:smooths_dense_in_a}
 \itemm{b} For $Y,Z\in\DS$ it holds
           $D_{t,x}(YZ) = Y D_{t,x}Z + Z D_{t,x}Y + x D_{t,x} Y D_{t,x} Z$ $\mm\otimes\PP$-a.e.
\end{itemize}

\end{lemma}

\begin{proposition}\label{proposition:first_double_inequality}
 Let $Y = \sum_{n=0}^\infty I_n(f_n)$ be bounded and $A\in\bor(\R_+\times\R_0)$ be such that $ \E\left[N(A)\right] =  (\ud t \otimes \nu)(A) < \infty$.
 Then $\sum_{n=1}^\infty nI_{n-1}\brac \tilde{f_n}(\cdot,{ * })\kets \one_A({ *})$ converges in $L_2(\mm\otimes\PP)$ and
 \begin{align}\label{equation:first_double_inequality}
  &\left| \normP{Y\sqrt{N(A)}} - \normP{Y} \sqrt{\E\sqr N(A)\brackets}  \right| 
   \leq \normmP{\sum_{n=1}^\infty  \left(   nI_{n-1}\!\brac \tilde{f_n} \kets \one_A  \right) } \nonumber \\*
  &\qquad\qquad\leq \normP{Y\sqrt{N(A)}} + \normP{Y}\sqrt{\E\sqr N(A)\brackets} .
 \end{align}
\end{proposition}
\begin{proof}
 Assume first that $Y\in\DS$. Then also $Y^2 = \sum_{n=0}^\infty I_n(g_n) \in\DS$. 
Letting $h(t,x) = \frac{1}{x}\one_A(t,x)$ we have $I_1(h) = N(A) - \E\sqr N(A) \brackets$ and we get using
\eqref{equation:inner_product_L_2} and \eqref{equation:inner_product_DD} that
 \begin{align*}
 \E\sqr Y^2  N(A)\brackets - \E\sqr Y^2 \brackets \E\sqr N(A)\brackets
 & = \E\sqr Y^2 I_1(h) \brackets   
  = (g_1,h)_{L_2(\mm)}\\
 & = (DY^2,h{\one_{\Om}})_{L_2(\mm\otimes\PP)}.
\end{align*}
From Lemma \ref{lemma:smooths_dense_in} (b) we obtain
\begin{align*}
    \E\sqr Y^2  N(A)\brackets
 & = \E\sqr Y^2 \brackets \E\sqr N(A)\brackets + (DY^2,h)_{L_2(\mm\otimes\PP)}\\
 & = \E\sqr Y^2 \brackets \E\sqr N(A)\brackets + 2  \iint_A \E\sqr YD_{t,x}Y \brackets x\ud t \nu(\ud x)\\*
 &\qquad + \iint_A \E\sqr \brac D_{t,x} Y\kets^2\brackets\mm(\ud t,\ud x).
 \end{align*}
 Using H\"{o}lder's inequality we get
 $$\left| 2  \iint_A \E\sqr YD_{t,x}Y \brackets x\ud t \nu(\ud x)\right| \leq 2\normP{Y}\sqrt{\E\sqr N(A)\brackets} \normmP{DY\one_A},$$
 so that
 \begin{align*}
   &  \brac -\normP{Y}\sqrt{\E\sqr N(A)\brackets} + \normmP{DY\one_A} \kets^2
    \leq \E \sqr Y^2 N(A)\brackets\\*
   & \qquad\qquad\leq \brac \normP{Y}\sqrt{\E\sqr N(A)\brackets} + \normmP{DY\one_A} \kets^2.\end{align*}
 Taking the square root yields to the double inequality \eqref{equation:first_double_inequality}.

 Using Lemma \ref{lemma:smooths_dense_in} (a) we find for any bounded $Y$ a uniformly bounded sequence $(Y_k)\subset \DS$ such that 
 $Y_k\to Y$ a.s.  Since inequality \eqref{equation:first_double_inequality} holds for all random variables 
 $Y_k-Y_m \in \mathcal{S}$, they are uniformly bounded  and $Y_k-Y_m\to 0$ a.s. as $k,m\to\infty$, we have by dominated convergence that
  \begin{align*}
   & \normmP{D(Y_k-Y_m) \one_A}\\
   &\leq \normP{(Y_k-Y_m)\sqrt{N(A)}} + \normP{Y_k-Y_m}\sqrt{\E\sqr N(A)\brackets} \\ 
   &\to 0  
 \end{align*}
 as $k,m\to\infty$.
 Thus the sequence $(DY_{k}\one_A)_{k=1}^\infty$ converges
 in $L_2(\mm\otimes\PP)$ to some mapping $u\in L_2(\mm\otimes\PP)$. Write $Y_k=\sum_{n=0}^\infty I_n \left( \tilde{f}_n^{(k)}  \right)$.
 The mapping $u$ has a representasion $u = \sum_{n=0}^\infty I_n(h_{n+1})$
 (see Remark \ref{remark:inner_product_DD}), where for all $n\geq0$ we have that
 \begin{align*}
  \left\| n\tilde{f_n}\one_A - h_n \right\|_{L_2({\mm^{\otimes n}})}  
 & \leq  \left\| n\tilde{f_n}\one_A - n\tilde{f}_n^{(k)}\one_A \right\|_{L_2({\mm^{\otimes n}})} 
    \!+ \left\| n\tilde{f}_n^{(k)}\one_A - h_n \right\|_{L_2({\mm^{\otimes n}})} \\*
    & \to 0 
  \end{align*} 
   as $k\to\infty$.
  We obtain \eqref{equation:first_double_inequality} for the random variable $Y$ using dominated convergence, the convergence
  $DY_k\one_A \to  \sum_{n=0}^\infty \left( D I_n(f_n) \one_A \right)$ in $L_2(\mm\otimes\PP)$ and the fact that 
  \eqref{equation:first_double_inequality} holds for all random variables $Y_{k}$. 
\end{proof}

\begin{lemma}\label{lemma:Lipschitz}
 If $Y=\sum_{n=0}^\infty I_n(f_n\one_{\R_+\times\R_0}^{\otimes n}) \in \DD$ and $g:\R\to\R$ is Lipschitz-continuous, then
 $g(Y)\in\DD$ and
 $$D_{t,x}g(Y) = \frac{g(Y + x D_{t,x}Y) - g(Y)}{x} \quad \text{ in } L_2(\mm\otimes\PP).$$ 
\end{lemma}
\begin{proof}
The lemma is an immediate consequence of \cite[Lemma 5.1 (b)]{geiss-laukkarinen}.
\end{proof}

\begin{lemma}\label{lemma:measurability}
Let $Y = \sum_{n=0}^\infty I_n(f_n)\in L_2(\PP)$ and $A\in\bor(\R_+\times\R)$. Then
$$\E \sqr Y | \ftn_A \brackets = \sum_{n=0}^\infty I_n \brac f_n\one_A^{\otimes n} \kets \text{ in }L_2(\PP).$$
\end{lemma}
\begin{proof}
The equality can be shown via the construction of the chaos analogously to the proof of \cite[Lemma 1.2.4]{nualartv1}.
\end{proof}

\begin{proof}[Proof of Theorem \ref{theorem:second_double_inequality}]
{ \emph{1.}} Assume $Y \in\DD$ and define $g_m(x)= (-m \vee x)\wedge m$ for $m\geq 1$. From Lemma \ref{lemma:Lipschitz} we get 
$g_m(Y)\in\DD$ and $|Dg_m(Y)|\leq |DY|$.  Then, using monotone convergence and Proposition
\ref{proposition:first_double_inequality}, we obtain
\begin{align*}
& \left| \normP{Y\sqrt{N(A)}} - \normP{Y} \sqrt{\E\sqr N(A)\brackets}  \right| \\
& = \lim_{m\to\infty} \left| \normP{g_m(Y)\sqrt{N(A)}} - \normP{g_m(Y)} \sqrt{\E\sqr N(A)\brackets}  \right| \\
&\leq \limsup_{m\to\infty} \normmP{Dg_m(Y)\one_A}\\
& \leq  \normmP{DY\one_A} < \infty.
\end{align*}
Hence $Y\sqrt{N(A)}\in L_2(\PP)$.

{ \emph{2.}} Assume $\| Y\sqrt{N(A)}\| < \infty$ and define $g_m(Y)$ as above. 
Write $Y = \sum_{n=0}^\infty I_n \brac f_n \kets$ and $g_m(Y) = \sum_{n=0}^\infty I_n ( f_n^{(m)})$. 
Since $g_m(Y)\to Y$ in $L_2(\PP)$, it holds 
$\| \tilde{f}_n^{(m)} \|^2_{L_2(\mm^{\otimes n})}\to \| \tilde{f_n} \|^2_{L_2(\mm^{\otimes n})}$ as $m\to\infty$.
 Since $g_m(Y)$ is $\ftn_A$-measurable, we have $\tilde{f}_n^{(m)} = \tilde{f}_n^{(m)}\one_A^{\otimes n}$ $\mm^{\otimes n}$-a.e. 
by Lemma \ref{lemma:measurability} for all $m\geq 1$. 
 By Fatou's Lemma, 
Proposition \ref{proposition:first_double_inequality} and monotone convergence we get 
\begin{align*}
  & \sqrt{ \sum_{n=1}^\infty nn! \left\| \tilde{f_n} \right\|^2_{L_2(\mm^{\otimes n})}}  \\*
   &\leq \liminf_{m\to\infty}\sqrt{\sum_{n=1}^\infty nn! \left\| \tilde{f}_n^{(m)} \right\|^2_{L_2(\mm^{\otimes n})}} \\*
   &\leq \liminf_{m\to\infty} \brac \normP{g_m(Y)\sqrt{N(A)}} + \normP{g_m(Y)} \sqrt{\E\sqr N(A)\brackets}  \kets\\*
  & = \normP{Y\sqrt{N(A)}} + \normP{Y} \sqrt{\E\sqr N(A)\brackets} < \infty.
\end{align*}
Thus $Y\in\DD$.
\end{proof}

We use the notation $ \alpha \sim_c  \beta$  for $\frac{1}{c}  \beta  \leq  \alpha \leq c \beta$ for $c\geq 1$ and 
$\alpha, \beta\in[0,\infty]$.
 
\begin{corollary}\label{corollary:norm_equivalence}
Let $A\in\bor(\R_+\times\R_0)$ be such that $ \E \left[ N(A)\right]  < \infty$ and assume that 
$Y= \sum_{n=0}^\infty I_n(f_n)\in L_2(\PP)$
 is $\ftn_A$-measurable.  Then
$$  \|Y\|_{\DD} \sim_{\sqrt{2}\left( \sqrt{ \E \sqr N(A) \brackets} +1  \right)}  
\left\|   Y  \sqrt{N(A) +1 }  \right\|_{L_2(\PP)}, $$
 where the norms may be infinite. 
\end{corollary}
\begin{proof}
The inequalities \eqref{equation:second_double_inequality1} and \eqref{equation:second_double_inequality2} give  the relation
$$\left \| Y \sqrt{N(A)} \right\|_{L_2(\PP)} + \|Y\|_{L_2(\PP)} \sim_{  \sqrt{ \E \sqr N(A) \brackets} +1 } \|Y\|_{L_2(\PP)} 
+ \|DY\|_{L_2(\mm\otimes\PP)}. $$ 
The claim follows from $\|Y\|_{\DD} \leq \|Y\|_{L_2(\PP)} + \|DY\|_{L_2(\mm\otimes\PP)} \leq \sqrt{2}\|Y\|_{\DD} $ and
\begin{align*}
       { \normP{Y\sqrt{N(A)+1}} }
 & \leq \left\| Y \brac \sqrt{N(A)} + 1 \kets    \right\|_{L_2(\PP)} \\
 & \leq \left\| Y \sqrt{N(A)}  \right\|_{L_2(\PP)} + \|Y\|_{L_2(\PP)}\\
 & \leq { \sqrt{ 2 \left( \left\| Y \sqrt{N(A)}  \right\|^2_{L_2(\PP)} + \|Y\|^2_{L_2(\PP)} \right) } }\\
 & = { \sqrt{2} \left\|  Y  \sqrt{N(A) +1 }   \right\|_{L_2(\PP)}. } 
\end{align*}
\end{proof}



\section{Fractional differentiability}
\label{section:fractional}

We consider fractional smoothness in the sense of real interpolation spaces between $L_2(\PP)$ and $\DD$. 
For parameters $\theta\in(0,1)$ and $q\in[1,\infty]$ the interpolation space $(L_2(\PP),\DD)_{\theta,q}$
is a Banach space, intermediate between $L_2(\PP)$ and $\DD$.

We shortly recall the $K$-method of real interpolation. 
The K-functional of $Y\in L_2(\PP)$ is the mapping
$K(Y,\cdot; L_2(\PP),\DD): (0,\infty) \to [0,\infty)$ defined by 
\begin{align*} 
&K(Y,s; L_2(\PP),\DD)\\ & := \inf\{ \|Y_0\|_{L_2(\PP)} + s\|Y_1\|_{\DD}:\ Y=Y_0+Y_1,\,Y_0\in L_2(\PP),\, Y_1\in \DD \}
\end{align*}
and we shall use the abbreviation $K(Y,s)$ for $K(Y,s; L_2(\PP),\DD)$.
Let $\theta\in(0,1)$ and $q\in[1,\infty]$. The space $(L_2(\PP),\DD)_{\theta,q}$ consists of all $Y\in L_2(\PP)$
such that
\[
\|Y\|_{(L_2(\PP),\DD)_{\theta,q}} 
= \begin{cases}
  \left[ \int_0^\infty \left| s^{-\theta} K(Y,s) \right|^q \frac{\mathrm{d}s}{s} \right]^{\frac{1}{q}}, & q\in[1,\infty)\\
  \sup_{s>0} s^{-\theta} K(Y,s), & q=\infty
                               \end{cases}
\]
is finite.

The interpolation spaces are nested in a lexicographical order:
\[
\DD \subset (L_2(\PP),\DD)_{\eta,p} \subset (L_2(\PP),\DD)_{\theta,q} \subseteq (L_2(\PP),\DD)_{\theta,p} \subset L_2(\PP)
\]
 for $ 1 \leq q \leq p \leq \infty$ and
 $ 0 < \theta < \eta < 1$.
For further properties of interpolation we refer to \cite{bennet-sharpley} and \cite{triebel}.

\begin{theorem} \label{theorem:fractional}
 Let $\theta\in(0,1)$, $A\in\bor(\R_+\times\R_0)$ be such that $ \E\left[ N(A) \right]< \infty$ 
 and $Y\in L_2(\PP)$
 be $\ftn_A$-measurable. Then 
$$Y \in (L_2(\PP),\DD)_{\theta,2}\text{ if and only if }  \E\left[ Y^2 N(A)^\theta  \right] < \infty.$$ If
 $Y \in (L_2(\PP),\DD)_{\theta,2}$, then
$$ \|Y\|_{(L_2(\PP),\DD)_{\theta,2}}
 \sim_{\sqrt{2}\frac{  \sqrt{ \E\sqr N(A) \brackets } +1}{\sqrt{\theta(1-\theta)}}}  \left\| Y  \sqrt{N(A) +1}^{\,\theta}  \right\|_{L_2(\PP)}.$$
\end{theorem}
\begin{proof}
We first show that
 \begin{equation} \label{equation:Kfunctional}
  K(Y,s)    
\sim_{ 2\brac  \sqrt{\E\sqr N(A)\brackets} +1  \kets} \left\|  Y \min\left\{1, s  \sqrt{N(A) +1}  \right\}  \right \|_{L_2(\PP)}.
  \end{equation}
From Lemma \ref{lemma:measurability} we obtain the inequalities $  \| \E \sqr Y_0 | \ftn_A \brackets \|_{L_2(\PP)} \leq \|Y_0\|_{L_2(\PP)} $
and  $\| \E \sqr Y_1 | \ftn_A \brackets \|_{\DD} \leq \|Y_1\|_{\DD} $ for any $Y_0\in L_2(\PP)$ and $Y_1\in\DD$. Hence
\begin{align}\label{align:K_relation_one}
     & K(Y,s) \nonumber\\*
& =  \inf\left\{ \|Y_0\|_{L_2(\PP)} + s\|Y_1\|_{\DD}:\ Y_0+Y_1 = Y,\, Y_0\in L_2(\PP),\, Y_1\in \DD \right\}\nonumber \\
& = \inf\left\{ \| \E \sqr Y_0 | \ftn_A \brackets \|_{L_2(\PP)} + s\| \E \sqr Y_1 | \ftn_A \brackets \|_{\DD}: Y_0+Y_1 = Y,\, Y_1\in \DD \right\}\nonumber \\*
& \sim_{c} 
   \inf\left \{ \|Y_0\|_{L_2(\PP)} +  s \left\|  Y_1   \sqrt{N(A)+1}   \right\|_{L_2(\PP)}  : 
     Y_0 + Y_1 = Y, Y_1\in\DD \right \}
\end{align}
for $c = \sqrt{2}\brac \sqrt{\E\sqr N(A) \brackets} +1\kets $ by Corollary \ref{corollary:norm_equivalence}. 
Next we approximate the $K$-functional from above with the choice $Y_0 = Y\one_{\left\{  \sqrt{N(A)+1}  > \frac{1}{s}\right\}} $ 
and get from \eqref{align:K_relation_one}  that
\begin{align*}
&     \frac{1}{c}  K(Y,s)\\
& \leq   \brac \left\|Y \one_{ \left\{{ \sqrt{N(A) +1}} > \frac{1}{s} \right\}} \right\|_{L_2(\PP)} 
       + s \left\| Y{ \sqrt{N(A) +1} }  \one_{ \left\{ { \sqrt{N(A) +1 }} \leq \frac{1}{s} \right\}} \right\|_{L_2(\PP)} \kets \\
& \leq \sqrt{2}  \left\| Y \min\left\{ 1, s { \sqrt{N(A) +1 }} \right\}  \right\|_{L_2(\PP)}.
\end{align*}
Using the triangle inequality and the fact that $$|Y(\om) -y| + |y|a \geq |Y(\om)| \min\{1,a\}$$ for all $\om\in\Om$, 
$y\in\R$ and $a\geq 0$  we  obtain from \eqref{align:K_relation_one} the lower bound 
\begin{align*}
  &   c  K(Y,s)\\*
& \geq \inf \left\{ \left\| |Y_0| + |Y_1|s { \sqrt{N(A) +1} }   \right\|_{L_2(\PP)} : Y = Y_0+Y_1,\, Y_1\in\DD  \right\}\\*
& \geq \left\| Y \min \left\{ 1, s { \sqrt{N(A) +1}}  \right\}  \right\|_{L_2(\PP)}.
\end{align*}
We have shown that \eqref{equation:Kfunctional} holds.
From \eqref{equation:Kfunctional} we get 
\begin{align*}
& \|Y\|_{(L_2(\PP),\DD)_{\theta,2}} \\*
&\sim_{2\brac \sqrt{\E \sqr N(A) \brackets} +1 \kets} 
\brac \int_0^\infty \left| s^{-\theta} \normP{ Y \min\left\{ 1, s  { \sqrt{N(A) +1 } }  \right\}  } \right|^2 
\frac{\ud s}{s} \kets^{\frac{1}{2}}.
\end{align*}
We finish the proof by computing the integral using first Fubini's theorem. We get
\begin{align*}
&   \int_0^\infty \left| s^{-\theta} \normP{ Y \min\left\{ 1, s { \sqrt{N(A) +1 } }  \right\} } \right|^2 \frac{\ud s}{s}\\*
& =  \E \sqr Y^2 \int_0^\infty s^{-2\theta}  \min\left\{ 1, s^2 {\brac N(A) +1 \kets } \right\}\frac{\ud s}{s}  \brackets  \\*
& =  \E \sqr Y^2 \frac{1}{2\theta(1-\theta)} { \brac N(A) + 1 \kets^{\theta} } \brackets.
\end{align*}
\end{proof}



\section{Concluding remarks}

From Theorem \ref{theorem:second_double_inequality} assertion \emph{2.} we can conclude that a higher integrability than square integrability
can imply Malliavin differentiability. For example,
all the spaces $L_p(\Om,\ftn_A, \PP)$ are subspaces of $\DD$ when $p>2$ and $\E[N(A)]< \infty$ as we can deduce from the following corollary.

\begin{corollary} Let $A\in\bor(\R_+\times\R_0)$ be such that $\lambda:=\E[N(A)] \in (0, \infty)$ so that $N(A)\sim$ Poisson($\lambda$). 
Then for the space
$$L_2 \log^+ L_2(\Om,\ftn_A,\PP) : = \left\{ Y \in L_2(\Om,\ftn_A,\PP): \E \sqr Y^2 \ln^+ Y^2  \brackets < \infty \right\},$$
where $\ln^+x = \max\{\ln x , 0\}$, it holds that
$$L_2 \log^+ L_2(\Om,\ftn_A,\PP)\subsetneq \DD \cap L_2(\Om,\ftn_A,\PP).$$
\end{corollary}
\begin{proof}
Suppose $\E \sqr Y^2 \ln^+ Y^2  \brackets < \infty$  and let $\vph(y)=\ln(y+1)$. The functions $\Phi$ and $\Phi^\star$ with
$$\Phi(x) = \int_0^x \vph(y) \ud y = (x+1)\ln(x+1) -x  \leq 1 + x\ln^+ x$$ 
and
$$\Phi^\star(x)= \int_0^x \vph^{-1}(y) \ud y =e^x - x -1$$
are a complementary pair of Young functions. They satisfy the
Young inequality $xy \leq \Phi(x) + \Phi^\star(y)$ for all $x,y\geq0$ and we get
\begin{align*}
       \E \left[ Y^2 N(A)  \right] 
& \leq \E \sqr \Phi\left( Y^2 \right) \brackets + \E \sqr \Phi^\star(N(A)) \brackets   \\*
& \leq \E \left[  Y^2  \ln^+\left( Y^2 \right)   \right] + e^{(e-1)\la } - \la   \\*
 &< \infty. \end{align*}
Hence $Y\in\DD$ by Theorem \ref{theorem:second_double_inequality}.

To see that the inclusion is strict, let $a\in(1,2]$ and choose a Borel function $f:\mathbbm{R}\to\mathbbm{R}$ such that
$f(0)=f(1)=0$ and
$$f(n) = \sqrt{e^\lambda \frac{n!}{\lambda^n} \frac{1}{n^2 \ln^a n}} \quad \text{ for }n = 2,3,\ldots.$$
Then, since
 $\ln n!  = \sum_{k=2}^n \ln k \geq \int_1^n \ln x\, \ud x  = n\ln n - n  + 1 $ for $n\geq 2$ and $a\leq 2$, we have
\begin{align*}
  \E\left[f^2(N(A))\ln^+ f^2(N(A))\right] 
& = \sum_{n=2}^\infty  \frac{1}{n^2 \ln^a n}\ln \left( e^\lambda \frac{n!}{\lambda^n} \frac{1}{n^2 \ln^a n}  \right)\\
& = \sum_{n=2}^\infty  \frac{\ln n!}{n^2 \ln^a n}
    + \sum_{n=2}^\infty  \frac{1}{n^2 \ln^a n}\ln \left(e^\lambda\frac{1}{\lambda^n} \frac{1}{n^2 \ln^a n}  \right) \\
& = \infty,
\end{align*}
but
$$\E\left[N(A) f^2(N(A))\right] = \sum_{n=2}^\infty nf^2(n) e^{-\lambda} \frac{\lambda^n}{n!} = \sum_{n=2}^\infty \frac{1}{n\ln^a n} < \infty $$
so that $f(N(A)) \in \DD$ by Theorem \ref{theorem:second_double_inequality}.
\end{proof}


\begin{remark}\label{remark:compound_differentiability}
 Suppose $\s=0$ and $\nu(\R)<\infty$, which means that $X$ is a compound Poisson process (with drift) and
  $$X_t = \beta t + \int_{(0,t]\times\R_0} x N(\ud s, \ud x)\quad  \text{ for all }t\geq 0 \text{ a.s.}$$
 for some $\beta\in\R$. The process $(N_t)_{t\geq0}$, with $N_t = N((0,t]\times\R_0)$ a.s., is the Poisson process associated to $X$.
 Let $T\in(0,\infty)$ and $\ftn_T$ be the completion of the sigma-algebra generated by  $(X_t)_{t\in[0,T]}$. Then
  $\ftn_T = \ftn_{[0,T]\times\R}$ and by Theorems \ref{theorem:second_double_inequality} and \ref{theorem:fractional} for
 any
 $\ftn_T$-measurable random variable $Y$ and any $\theta\in(0,1)$ it holds that
 \begin{itemize} \item[(a)] $Y\in\DD$ if and only if
 $\normP{ Y\sqrt{N_T+1}}  < \infty$ and
 \item[(b)] $Y\in(L_2(\PP),\DD)_{\theta,2}$ if and only if $\normP{ Y\sqrt{N_T+1}^{\,\theta}}  < \infty$.
 \end{itemize}
\end{remark}

{\bf Acknowledgements.} The author is grateful to Christel Geiss and Stefan Geiss for several
valuable ideas and suggestions regarding this work.




\begin{thebibliography}{10}

\bibitem{alos-leon-vives}
E.~Al\'{o}s, J.~A. Leon, and J.~Vives.
\newblock An anticipating {I}t\^{o} formula for {L}\'{e}vy processes.
\newblock {\em Alea}, 4:285--305, 2008.

\bibitem{applebaum2}
D.~Applebaum.
\newblock Universal {M}alliavin calculus in {F}ock and {L}\'{e}vy-{I}t\^{o}
  spaces.
\newblock {\em Commun. Stoch. Anal.}, 3:119--141, 2009.

\bibitem{bennet-sharpley}
C.~Bennett and R.~Sharpley.
\newblock {\em Interpolation of operators}.
\newblock Academic Press, New York, 1988.

\bibitem{benth-dinunno-lokka-oksendal-proske}
F.E. Benth, G.~Di Nunno, A.~L{\o}kka, B.~{\O}ksendal, and F.~Proske.
\newblock Explicit representation of the minimal variance portfolio in markets
  driven by {L}\'evy processes.
\newblock {\em Math. Finance}, 13(1):55--72, 2003.

\bibitem{geiss-geiss}
C.~Geiss and S.~Geiss.
\newblock On approximation of a class of stochastic integrals and
  interpolation.
\newblock {\em Stoch. Stoch. Rep.}, 76:339--362, 2004.

\bibitem{geiss-geiss-laukkarinen}
C.~Geiss, S.~Geiss, and E.~Laukkarinen.
\newblock A note on {M}alliavin fractional smoothness for {L}\'{e}vy processes
  and approximation.
\newblock {\em Potential Anal.}, 39(3):203--230, 2013.

\bibitem{geiss-laukkarinen}
C.~Geiss and E.~Laukkarinen.
\newblock Denseness of certain smooth {L}\'{e}vy functionals in
  $\mathbbm{D}_{1,2}$.
\newblock {\em Probab. Math. Statist.}, 31(1):1--15, 2011.

\bibitem{geiss-steinicke}
C.~Geiss and A.~Steinicke.
\newblock Malliavin derivative of random functions and applications to {L}\'evy
  driven {BSDE}s.
\newblock {\em Electron. J. Probab.}, 21(10), 2016.

\bibitem{geiss-hujo}
S.~Geiss and M.~Hujo.
\newblock Interpolation and approximation in {$L_2(\gamma)$}.
\newblock {\em J. Approx. Theory}, 144, 2007.

\bibitem{ito}
K.~It\^{o}.
\newblock Spectral type of the shift transfomation of differential process with
  stationary increments.
\newblock {\em Trans. Amer. Math. Soc.}, 81:253--263, 1956.

\bibitem{lee-shih_product_formula}
Y.-J. Lee and H.-H. Shih.
\newblock The product formula of multiple {L}\'{e}vy-{I}t\^{o} integrals.
\newblock {\em Bull. Inst. Math. Acad. Sinica}, 32(2):71--95, 2004.

\bibitem{mecke}
J.~Mecke.
\newblock Station{\"a}re zuf{\"a}llige {M}asse auf lokalkompakten {A}belschen
  {G}ruppen.
\newblock {\em Z. Wahrscheinlichkeitstheorie und Verw. Gebiete}, 9(1):36--58,
  1967.

\bibitem{nualartv1}
D.~Nualart.
\newblock {\em The Malliavin Calculus and Related Topics}.
\newblock Springer, 1995.

\bibitem{nualart-vives}
D.~Nualart and J.~Vives.
\newblock Anticipative calculus for the {P}oisson process based on the {F}ock
  space.
\newblock {\em S\'{e}minaire de Probabilit\'{e}s}, 24:154--165, 1990.

\bibitem{privault_extension}
N.~Privault.
\newblock An extension of stochastic calculus to certain non-{M}arkovian
  processes.
\newblock Preprint~49, Universite d'Evry, 1997.
\newblock http://www.maths.unive-evry.fr/prepubli/49.ps.

\bibitem{sole-utzet-vives}
J.~Sol\'{e}, F.~Utzet, and J.~Vives.
\newblock Canonical {L}\'{e}vy process and {M}alliavin calculus.
\newblock {\em Stochastic Process. Appl.}, 117:165--187, 2007.

\bibitem{steinicke}
A.~Steinicke.
\newblock Functionals of a {L}\'evy process in canonical and generic
  probability spaces.
\newblock J. Theoret. Probab., 2014.

\bibitem{triebel}
H.~Triebel.
\newblock {\em Interpolation Theory, Function spaces, Differential Operators}.
\newblock North-Holland, 1978.

\end{thebibliography}
\end{document}